\documentclass[11pt]{article}
\usepackage{bbm}
\usepackage{amssymb}
\usepackage{amsfonts}
\usepackage{mathrsfs}
\usepackage{mathrsfs}
\usepackage{mathrsfs}
\usepackage{amssymb,amsmath,amsthm,amsfonts}
\usepackage{indentfirst}
\usepackage{color}
\usepackage[all]{xy}
\usepackage[numbers,sort&compress]{natbib}

\def\be{\begin{equation}}
\def\ee{\end{equation}}

\usepackage{graphicx}
\newtheorem{theorem}{Theorem}[section]                   %theory and  corollary
\newtheorem{definition}{Definition}[section]                   %definition
\newtheorem{proposition}{Proposition}[section]
\newtheorem{lemma}{Lemma}[section]
\newtheorem{corollary}{Corollary}[section]
\newtheorem{remark}{Remark}[section]

\begin{document}

\title{\LARGE \bf Subharmonic Solutions and Minimal Periodic Solutions of First-order Hamiltonian Systems with Anisotropic Growth}
\author{ Chungen Liu\thanks{Partially supported by the NSF of China (11071127, 10621101), 973 Program of
MOST (2011CB808002) and SRFDP. E-mail:
liucg@nankai.edu.cn}\ \ and Xiaofei Zhang\\\
{\small School of Mathematics and LPMC, Nankai University,}\\
{\small Tianjin 300071, P. R. China}\\
{\small }}
\date{}
 \maketitle
\noindent {\bf Abstract}: Using a homologically link theorem in variational theory and iteration inequalities of Maslov-type index,
we show the existence of  a sequence of subharmonic solutions of non-autonomous Hamiltonian systems with the Hamiltonian functions  satisfying some anisotropic growth conditions, {\it i.e.}, the Hamiltonian functions may have simultaneously, in different components, superquadratic, subquadratic and quadratic behaviors.
Moreover, we also consider the  minimal period problem of some autonomous Hamiltonian systems with anisotropic growth.

\vspace{0.3cm} \noindent {\bf Key Words:} \ Maslov-type index;\ Morse index;\ homologically link;\
Subharmonic solution;\ minimal period.

\section{Introduction}
\setcounter{equation}{0}

In this paper, we first consider subharmonic solutions of the following Hamiltonian system
\begin{equation}\label{1.1}
\left\{ \begin{array}{ll}-J\dot{z}=H^{\prime}_{z}(t,z),\\
z(k\tau)=z(0),\;\;k\in \mathbf{Z},\end{array}\right.
\end{equation}
where
$H^{\prime}_{z}$ is the gradient of $H$ with respect to the variables $z=(p_1,\cdots, p_{n}, q_1,\cdots, q_{n})\in \mathbf{R}^{2n}$ and
$J=\left(\begin{array}{cc}
0 & -I_n\\
I_n & 0\end{array}\right)$
with $I_{n}$ being the $n\times n$ identity matrix.

Denote any principal diagonal matrix diag$\{a_{1}, \cdots, a_{n}, b_{1}, \cdots, b_{n}\}\in\mathbf{R}^{2n}$
by $V(a,b)$ with $a=(a_1,\cdots, a_n)$ and $b=(b_1,\cdots, b_n)$,
then $V(a,b)(z)=(a_{1}p_{1},\cdots, a_{n}p_{n}, b_{1}q_{1},\cdots, b_{n}q_{n})$. Now
we suppose the Hamiltonian function $H$ satisfying  the following conditions
as in \cite{110} with a bit difference.

(H1) $H \in C^{2}(\mathbf{R} \times \mathbf{R}^{2n},\mathbf{R})$
is nonnegative and $\tau$-periodic with respect to $t$.

(H2) There exist $\beta>1$ and $c_{1}, c_{2}, \alpha_{i}, \beta_{i}>0$ with $\alpha_{i}+\beta_{i}=1$
$(i=1, 2, \cdots, n)$ such that
$$H^{\prime}_{z}(t,z)\cdot V_{1}(z)-H(t,z)\geq c_{1}|z|^{\beta}-c_{2}
,\ \ \ (t,z)\in \mathbf{R}\times\mathbf{R}^{2n},$$
where $V_{1}=V(\alpha,\beta)$.

(H3) There exist $ {\sigma}_{i},  {\tau}_{i}> 0$ and
$\lambda\in \Lambda$ with $\Lambda=\{\lambda\in \mathbf{R}:
\max_{1\leq i\leq n}\{\frac{ {\sigma}_{i}}{ {\tau}_{i}}, \frac{ {\tau}_{i}}{ {\sigma}_{i}}\}<\lambda<1+\beta\}$ such that
$$|H''_{zz}(t,z)|\leq c_{2}(|z|^{\lambda-1}+1),\ \ \ (t,z)\in \mathbf{R}\times\mathbf{R}^{2n}.$$

(H4) $\displaystyle\frac{H(t,z)}{\omega(z)}\rightarrow 0$ as
 $|z|\rightarrow 0$ uniformly in $t$, where $w(z)=\sum_{i=1}^{n}
 \left(|p_{i}|^{1+\frac{ {\sigma}_{i}}{ {\tau}_{i}}}+|q_{i}|^{1+\frac{ {\tau}_{i}}{ {\sigma}_{i}}}\right)$.

(H5) $\displaystyle\frac{H(t,z)}{w(z)}\rightarrow +\infty$ as
$|z|\rightarrow+\infty$ uniformly in $t$.\vspace{2mm}

Note that the Hamiltonian function
\begin{equation}\label{1.2}
H(t,z)=\displaystyle\sum_{i=1}^{n}\left(|p_i|^{1+\frac{ {\sigma}_i}{ {\tau}_i}}
\ln(1+p_i^2)+|q_i|^{1+\frac{ {\tau}_i}{ {\sigma}_i}}\ln(1+q_i^2)\right),\ \ \ \ (t,z)\in \mathbf{R}\times\mathbf{R}^{2n}
\end{equation}
is an example satisfying (H1)-(H5) with anisotropic growth when $ {\sigma}_i\ne  {\tau}_i$ for some $i$.
When $ {\sigma}_i= {\tau}_i$ for all $i$, it is an {\it almost quadratic growth}
function which is slower growing than any super-quadratic function at infinity in the sense of Ambrosetti-Rabinowitz.

Given $j\in \mathbf{Z}$ and a $k\tau$-periodic solution $(z, k\tau)$ of the system (\ref{1.1}), the phase shift $j\ast z$ of $z$
is defined by $j\ast z(t)=z(t+j\tau)$. Recall that two solutions $(z_{1},k_1\tau)$ and $(z_{2},k_2\tau)$ are geometrically distinct if
$j\ast z_{1}\neq l\ast z_{2}$, $j, l\in \mathbf{Z}$.

Now we list our main results of subharmonic solutions as following.

\begin{theorem}\label{1.20}
{Suppose $H$ satisfies (H1), (H2), (H4), (H5) and

(H3)$'$ there exist constants $\xi_{i}, \eta_{i}>0$
with $\xi_{i}+\eta_{i}=1$ $(i=1, 2, \cdots, n)$ such that
$$H^{\prime}_{z}(t,z)\cdot V_{2}(z)-H(t,z)\geq c_{1}|H_{z}'(t,z)|-c_{2}
,\ \ \ (t,z)\in \mathbf{R}\times\mathbf{R}^{2n},$$
where $V_{2}=V(\xi,\eta)$
and $\max_{1\leq i\leq n} \{\frac{\alpha_{i}}{\beta_{i}}, \frac{\beta_{i}}{\alpha_{i}},
\frac{\xi_{i}}{\eta_{i}}, \frac{\eta_{i}}{\xi_{i}},
\frac{{\sigma}_{i}}{{\tau}_{i}}, \frac{{\tau}_{i}}{{\sigma}_{i}}\}<1+\beta$.}

\noindent Then for each integer $k\geq 1$,
the system (\ref{1.1}) possesses a $k\tau$-periodic nonconstant solution $z_{k}$ such that
$z_{k}$ and $z_{pk}$ are geometrically distinct provided $p>2n+1$.
If all $z_{k}$ are non-degenerate, then
$z_{k}$ and $z_{pk}$ $(p>1)$ are geometrically distinct.
\end{theorem}

{Note that (H3)$'$ is weaker than (H3), so we have a similar result stated as a corollary of Theorem \ref{1.20}.}

\begin{corollary}\label{1.000} {Suppose $H$ satisfies (H1)-(H5),
 then we have the same results as in Theorem \ref{1.20}.}
\end{corollary}

\begin{theorem}\label{1.00}
The conclusions of Theorem \ref{1.20} still hold if $H$ satisfies the conditions

(C1) $H \in C^{2}(\mathbf{R}\times \mathbf{R}^{2n},\mathbf{R})$
is nonnegative and $\tau$-periodic with respect to $t$,

(C2) there exist constants $0<\theta<1$, $R, \varphi_{i}, \psi_{i}>0$
with $\frac{1}{\varphi_{i}}+\frac{1}{\psi_{i}}=1$
$(i=1, 2, \cdots, n)$ such that, by setting
$V_{3}=V(\varphi^{-1},\psi^{-1})$, we have
$$\theta H^{\prime}_{z}(t,z)\cdot V_{3}(z)\geq H(t,z)>0
,\ \ \ (t,z)\in \mathbf{R}\times\mathbf{R}^{2n}\ \text{with}\ |z|\geq R,$$

(C3) there exist constants $b_{1}>0$ and $b_{2}>0$ such that
$$|H'_{z}(t,z)|\leq b_{1}H^{\prime}_{z}(t,z)\cdot V_{3}(z)
+b_{2},\ \ \ (t,z)\in \mathbf{R}\times\mathbf{R}^{2n},$$

(C4) $\displaystyle\frac{H(t,z)}
{\sum_{i=1}^{n}\left(|p_{i}|^{\varphi_{i}}+|q_{i}|^{\psi_{i}}\right)}\rightarrow 0$ as
 $|z|\rightarrow 0$ uniformly in $t$.
\end{theorem}

The above conditions (C1)-(C4) are similar to that of \cite{813}
with minor difference.

\begin{remark}\label{1.90}
In the case where $H(t,z)=\frac{1}{2}(\hat{B}(t)z,z)+\hat{H}(t,z)$ with
$\hat{B}(t)$ being a $\tau$-periodic, continuous symmetric  matrix function and $\hat{H}$ satisfying the conditions
as stated in {Theorems \ref{1.20}, \ref{1.00} and Corollary \ref{1.000}}, we also obtain the similar results with some restrictive conditions on $\hat{B}$ in Section 4 below. But compared with the results in \cite{102},
we note that the condition $B(t)$ is semi-positive-definite required in \cite{102} is not necessary here (see Remark \ref{5.4}).

In Section 5, we consider the minimal periodic problem of some autonomous Hamiltonian systems with the Hamiltonian functions $H(z)$ satisfying the anisotropic growth conditions as stated in {Theorems \ref{1.20}, \ref{1.00} and Corollary \ref{1.000}}. With the same tricks, we show that the critical points $(z,\tau)$ obtained from the homological link method in fact is the minimal periodic solution of the Hamiltonian systems provided the Hessian {$H_{zz}''(z)$ is positively definite for $z\in \mathbf{R}^{2n}\setminus \{0\}$}.
\end{remark}

In the pioneer work \cite{8}, Rabinowitz obtained a sequence of subharmonic solutions of the system (\ref{1.1}).
Since then, many papers were devoted to the study of subharmonic solutions (see \cite{2, 1, 21, 6, 4, 42, 102, 5, 108}). For the brake subharmonic solutions of Hamiltonian systems we refer to \cite{LC1, LC3}. For the $P$-symmetric subharmonic solutions of Hamiltonian systems we refer to \cite{LT2}.
 We note that all the results obtained in the references mentioned here are related with the Hamiltonian functions with superquadratic growth or subquadratic growth.

This paper is organized as follows, in Section 2, as preliminary we recall some notions about the  Maslov-type index theory and the iteration inequalities developed by Y.Long and the first author of this paper in \cite{112}. In this section
we also recall the homologically link theorem in \cite{100} from which we can find a critical point of the corresponding functional together with   index information. Under the conditions as in {Theorems \ref{1.20}, \ref{1.00} and Corollary \ref{1.000}}, we show that there is a homologically link structure for the functional.
In Section 3, we give a proof of {Theorems \ref{1.20}, \ref{1.00} and Corollary \ref{1.000}}. In Section 4, we consider the existence of subharmonic solutions of the Hamiltonian systems
 in the case where $H$ may contain a quadratic term as $H(t,z)=\frac{1}{2}(\hat{B}(t)z,z)+\hat{H}(t,z)$.
We consider minimal  periodic problem  for  the autonomous Hamiltonian systems in Section 5.

\section{Preliminaries}
\setcounter{equation}{0}

We first recall the notion of Maslov-type index and  some iteration estimates. We refer to \cite{100}, \cite{112} and \cite{105} for  details.

Denote by $Sp(2n)=\{M\in \mathcal{L}(\mathbf{R}^{2n})| M^{t}JM=J\}$
the $2n\times 2n$  symplectic group,
where $M^{t}$ denotes the transpose of $M$.
Define $\mathcal P(2n)=\{\gamma|\gamma\in C([0,\tau],Sp(2n)),\;\;\gamma(0)=I\}$.

For $\gamma\in \mathcal{P}(2n)$, according to \cite{100} and \cite{105},
there is a Maslov-type index theory which assigns to $\gamma$  a pair of integers
$$(i_{\tau},\nu_{\tau}):=(i_{\tau}(\gamma),\nu_{\tau}(\gamma))\in \mathbf{Z}\times \{0, 1, \cdots, 2n\},$$
where $i_{\tau}$ is the index part of $\gamma$ and $\nu_{\tau}$ is the nullity.

For  $\gamma\in \mathcal{P}(2n)$, define $\gamma^{k} : [0,k\tau]\rightarrow Sp(2n)$ by
$$\gamma^{k}(t)=\gamma(t-j\tau)\gamma(\tau)^{j},\ \ \ \ j\tau\leq t\leq (j+1)\tau,\ 0\leq j\leq k-1.$$

We denote the Maslov index of $\gamma^{k}$ on the interval $[0,k\tau]$ by
$(i_{k\tau},\nu_{k\tau}):=(i_{k\tau}(\gamma^{k}),\nu_{k\tau}(\gamma^{k}))$.

In the case of linear Hamiltonian systems
\begin{equation*}
\left\{\begin{array}{l}
 -J\dot{z}=B(t)z\\
z(0)=z(\tau),\\
\end{array}\right.
\end{equation*}
where $B(t)$ is a $\tau$-periodic, symmetric and continuous matrix function. Its fundamental solution is denoted by
$\gamma_{B}\in C([0,\tau],Sp(2n))$ with $\gamma(0)=I_{2n}$. The Maslov-type index $(i_{\tau}(B),\nu_{\tau}(B)):=(i_{\tau}(\gamma_B),\nu_{\tau}(\gamma_B))$  is also called the Maslov-type index of the matrix function $B(t)$.

If $z$ is a $\tau$-periodic solution of the system (\ref{1.1}), we denote by
$(i_{\tau}(z),\nu_{\tau}(z)):=(i_{\tau}(B),\nu_{\tau}(B))$  with $B(t)=H_{z}''(t,z(t))$.
The solution $z$ is non-degenerate if $\nu_{\tau}(z)=0$.

\begin{proposition}(\cite{102})\label{2.11}
If $z$ is a $k\tau$-periodic solution of the system (\ref{1.1}),
then $i_{k\tau}(j\ast z)=i_{k\tau}(z)$ and $\nu_{k\tau}(j\ast z)=\nu_{k\tau}(z)$ hold for $0\leq j\leq k$.
\end{proposition}

\begin{proposition}(\cite{102})\label{2.10}
For $k\in \mathbf{N}$, there holds
$$k(i_{\tau}+\nu_{\tau}-n)-n\leq i_{k\tau}\leq k(i_{\tau}+n)+n-\nu_{k\tau}.$$
\end{proposition}

\begin{proposition}(\cite{112})\label{2.12}
For $m\in \mathbf{N}$, there holds
$$m(i_{\tau}+\nu_{\tau}-n)+n-\nu_{\tau}
\leq i_{m\tau}\leq m(i_{\tau}+n)-n-(\nu_{m\tau}-\nu_{\tau}).$$
\end{proposition}

\begin{lemma}(\cite{100})\label{2.121}
Let $B(t)$ be a $\tau$-periodic, symmetric and continuous matrix function. Assume
$B(t)$ are positive for $t\in [0, \tau]$ and $B(t_{0})$ is strictly positive for
some $t_0\in [0, \tau]$. Then $i_{\tau}(B)\geq n$.
\end{lemma}

\begin{lemma}(\cite{121})\label{2.120}
Let $B(t)$ be a $\tau$-periodic, symmetric and continuous matrix function.
Assume for some $k\in \mathbf{N}$, there hold $i_{k\tau}(B)\leq n+1$, $i_{\tau}(B)\geq n$
and $\nu_{\tau}(B)\geq1$. Then $k=1.$
\end{lemma}

Now we introduce some concepts and results of Sobolev space theory.

Let $E=W^{\frac{1}{2},2}(S_{\tau},\mathbf{R}^{2n})=
\bigg\{z\in L^{2}(S_{\tau},\mathbf{R}^{2n})\big| \displaystyle\sum_{j\in \textbf{Z}}|j||a_{j}|^{2}<+\infty\bigg\}$,
where $S_{\tau}:=\mathbf{R}/\tau\textbf{Z}$, $z(t)=\sum_{k\in \textbf{Z}}\exp{(\frac{2k\pi t}{\tau}J)}a_{k},
a_{k}\in \mathbf{R}^{2n}$.

For $\zeta\in E$, $\zeta(t)=\sum_{k\in \textbf{Z}}\exp{(\frac{2k\pi t}{\tau}J)}b_{k},
b_{k}\in \mathbf{R}^{2n}$, the inner product on $E$ is
$$\langle z,\zeta\rangle=\tau(a_{0},b_{0})+\tau\sum_{k\in \textbf{Z}}|k|a_{k}\cdot b_{k}.$$

\begin{lemma} \label{l2.1}(\cite{9}) (Embedding Theorem) The space
$E$ compactly embeds into $L^{s}(S_{\tau},\mathbf{R}^{2n})$ $(s\geq 1)$,
in particular, there exists a constant $C_{s}>0$
such that $\| z\|_{L^{s}}\leq C_{s}\| z\|$ holds for $z\in E$,
where $\|\cdot\|$ denotes the norm on $E$.
\end{lemma}

There exists a linear bounded self-adjoint operators $A$ such that
$\langle Az, \zeta\rangle=2\pi\sum_{k\in \textbf{Z}}k a_{k}\cdot b_{k}$,
obtained by extending the bilinear form $\langle Az, \zeta\rangle
=\int^{2\pi}_{0}(-J\dot{z},\zeta)\text{d}t$, $z, \zeta\in W^{1,2}(S_{\tau},\mathbf{R}^{2n})$.

Set $E^{\pm}= \{z\in E | z(t)=\sum_{\pm k>0}\exp{(\frac{2k\pi t}{\tau}J)}a_{k}, a_{k}\in \mathbf{R}^{2n}\}$
and $E^{0}= \mathbf{R}^{2n}$, then $Az^{\pm}=\pm\frac{2\pi}{\tau}z^{\pm}$, $z^{\pm}\in E^{\pm}$ (see \cite{102}).
Moreover, we set $E_{m}=\{z\in E | z(t)=\sum_{k=-m}^{m}\exp{(\frac{2k\pi t}{\tau}J)}a_{k},
a_{k}\in \mathbf{R}^{2n}\}$ and
$E_{m}^{\pm}=E^{\pm}\bigcap E_{m}$, and let $P_{m}$ be the corresponding orthogonal projection.

For $d>0$, we denote by $M_{d}^{+}(C)$, $M_{d}^{-}(C)$ and $M_{d}^{0}(C)$ the eigenspaces
of any linear bounded self-adjoint Fredholm operator $C$ corresponding to the eigenvalue $\lambda$
belonging to $(d,+\infty)$, $(-\infty,d)$ and $[-d,d]$ respectively.

Given $B(t)$ a $\tau$-periodic, symmetric and continuous matrix function with Maslov-type index
$(i_{\tau}(B),\nu_{\tau}(B))$, define $\langle Bz, \zeta\rangle=\int^{\tau}_{0}(B(t)z,\zeta)\text{d}t$, $z, \zeta\in E$.
Set $(A-B)^{\sharp}:= (A-B|_{R(A-B)})^{-1}$, then we have the following theorem.
\begin{lemma}(\cite{101})\label{2.2}
Suppose $0<d<\frac{1}{4}\| (A-B)^{\sharp}\|^{-1}$, then for $m$ large enough, there hold
\begin{equation*}\label{2.3}
 \text{dim}M_{d}^{+}(P_{m}(A-B)P_{m})=\frac{1}{2}\text{dim}(P_{m}E)-i_{\tau}(B)-\nu_{\tau}(B),
\end{equation*}
\begin{equation*}\label{2.4}
 \text{dim}M_{d}^{-}(P_{m}(A-B)P_{m})=\frac{1}{2}\text{dim}(P_{m}E)+i_{\tau}(B),
\end{equation*}
\begin{equation*}\label{2.5}
 \text{dim}M_{d}^{0}(P_{m}(A-B)P_{m})=\nu_{\tau}(B).
\end{equation*}
\end{lemma}

For $e\in E_{1}\bigcap E^{+}$ with $\| e\|=1$,
set $W=\big\{z\in \text{span}\{e\}\bigoplus E^{-}\bigoplus E^{0}
\big|1\leq\| z\|\leq2,
\| z^{-}\| \leq \| z^{+}+z^{0}\|\big\}$, then we have

\begin{lemma}(\cite{3})\label{3.130}
There exists a constant $\varepsilon_{1}>0$ such that
$$\emph{measure}\left\{t\in[0,T]\big| |z(t)|\geq \varepsilon_{1}\right\}\geq \varepsilon_{1},\ \ \ \ z\in W.$$
\end{lemma}

Finally, we recall the homologically link theorem in \cite{100}.

\begin{definition}(\cite{120})\label{2.07}
Let $Q$
be a topologically embedded closed $q$-dimensional ball on a Hilbert manifold $M$
and let $S\subset M$ be a closed subset such that
$\partial Q\bigcap S=\emptyset$. We say that $\partial Q$ and $S$ homotopically link if $\varphi(Q)\bigcap S\neq\emptyset$ for
$\varphi\in C(Q, M)$ with $\varphi|_{\partial Q}=\emph{id}|_{\partial Q}$.
\end{definition}

\begin{definition}(\cite{100})\label{2.7}
Let $Q$
be a topologically embedded closed $q$-dimensional ball on a Hilbert manifold $M$
and let $S\subset M$ be a closed subset such that
$\partial Q\bigcap S=\emptyset$. We say that $\partial Q$ and $S$ homologically link if $\partial Q$ is the
support of a non-vanishing homology class in $H_{q-1}(M\backslash S)$.
\end{definition}

\begin{lemma}\label{2.6}
Let $M=M_{1}\bigoplus M_{2}$ be a Hilbert space with dim$M_{2}=q-1$, $S=\partial B_{\mu}\bigcap M_{1}$
and $Q=(\bar{B}_{\nu}\bigcap M_{2})\bigoplus [0,\nu]e$, where $e\in M_{1}$ with $\|e\|=1$ and $\nu>\mu>0$. Let $B_{\mu}$, $B_{\nu}$
be two bounded linear invertible operators on $M$
such that $\nu>\mu\| B_{\nu}^{-1} B_{\mu}\|$ and
$PB_{\mu}^{-1}B_{\nu}: M_{2}\rightarrow M_{2}$ is invertible, where
$P: M\rightarrow M_{2}$ is the orthogonal projection. Then
$B_{\nu}(\partial Q)$ and $B_{\mu}(S)$ homologically link.
\end{lemma}

\begin{proof}
It is easy to prove $B_{\nu}(\partial Q)$ and $B_{\mu}(S)$ homotopically link (see \cite{41}).

Indeed, to show $B_{\nu}(Q)\bigcap B_{\mu}(S)\neq\emptyset$,
it is equivalent to proving
$\psi_{0}(t,v)=(\mu,0)$ has a solution in $[0,\nu]\times (\bar{B}_{\nu}\bigcap M_{2})$, where
$$\psi_{0}(t,v)=(\|B_{\mu}^{-1}B_{\nu}(te+v)\|,P_{2}B_{\mu}^{-1}B_{\nu}(te+v)), \ \ \ \ (t,v)\in
[0,\nu]\times (\bar{B}_{\nu}\bigcap M_{2}).$$

Note that $t=\mu\|B_{\mu}^{-1}B_{\nu}e-B_{\mu}^{-1}B_{\nu}B_{0}^{-1}B_{0}e\|^{-1}$ $(<\nu)$ and
$v=-tB_{0}^{-1}B_{0}e$ $(\in B_{\nu}\bigcap M_{2})$ is the unique solution of $\psi_{0}$
in $[0,\nu]\times (\bar{B}_{\nu}\bigcap M_{2})$
, where $B_{0}=P_{2}B_{\mu}^{-1}B_{\nu}$ and
$B_{0}^{-1}$ denotes the inverse of $B_{0}|_{M_{2}}$.
Thus $(\mu,0)\notin \psi_{0}(\partial ([0,\nu]\times (\bar{B}_{\nu}\bigcap M_{2})))$,
deg$(\psi_{0}, (0,\nu)\times (B_{\nu}\bigcap M_{2}), (\mu,0))=\pm 1$
and $B_{\nu}(\partial Q)\bigcap B_{\mu}(S)=\emptyset$.

For $\varphi\in C(B_{\nu}(Q), M)$ with $\varphi|_{B_{\nu}(\partial Q)}=\emph{id}|_{B_{\nu}(\partial Q)}$,
define $\psi: [0,\nu]\times (\bar{B}_{\nu}\bigcap M_{2})\rightarrow \mathbf{R}\times M_{2}$ as
$$\psi(t,v)=(\|B_{\mu}^{-1}\varphi B_{\nu}(te+v)\|,P_{2}B_{\mu}^{-1}\varphi B_{\nu}(te+v)).$$

In order to show $\varphi(B_{\nu}(Q))\bigcap B_{\mu}(S)\neq\emptyset$, it is equivalent to proving
$\psi(t,v)=(\mu,0)$ has a solution in $[0,\nu]\times (\bar{B}_{\nu}\bigcap M_{2})$.

Since $\psi=\psi_{0}$ on $\partial ([0,\nu]\times (\bar{B}_{\nu}\bigcap M_{2}))$, by the Brouwer
degree theory, deg$(\psi, (0,\nu)\times (B_{\nu}\bigcap M_{2}), (\mu,0))$=
deg$(\psi_{0}, (0,\nu)\times (B_{\nu}\bigcap M_{2}), (\mu,0))=\pm 1$, then
$\psi(t,v)=(\mu,0)$ has a solution in $[0,\nu]\times (\bar{B}_{\nu}\bigcap M_{2})$.
Hence, $B_{\nu}(\partial Q)$ and $B_{\mu}(S)$ homotopically link.

From Theorem \uppercase\expandafter{\romannumeral2}.1.2
in \cite{120}, we see $B_{\nu}(\partial Q)$ and $B_{\mu}(S)$ homologically link.
\end{proof}

Let $f$ be a $C^{2}$ functional on a Hilbert manifold $M$.
Recall that the Morse index $m(x)$ of $f$ at a critical point $x$ is the
dimension of a maximal subspace on which $D^{2}f(x)$ is strictly negative definite,
while the large Morse index $m^{*}(x)$ is $m(x)$+dim Ker$D^{2}f(x)$.

\begin{lemma}(\cite{100})\label{2.1}
Let $f$ be a $C^{2}$ functional on a Hilbert manifold $M$ with Fredholm gradient. Let $Q\subset M$
be a topologically embedded closed $q$-dimensional ball and let $S\subset M$ be a closed subset such that
$\partial Q\bigcap S=\emptyset$. Assume that $\partial Q$ and $S$ homologically link. Moreover, assume

(i) $\sup_{\partial Q} f<\inf_{S} f$,

(ii) $f$ satisfies (PS) condition on some open interval containing $[\inf_{S} f,\sup_{Q} f]$.

\noindent Then, if $\Gamma$ denotes the set of all $q$-chains in $M$ whose boundary has support $\partial Q$,
the number $c=\inf_{\xi\in \Gamma}\sup_{|\xi|} f\in [\inf_{S} f,\sup_{Q} f]$ is a critical value of $f$,
where $|\xi|$ denotes the support of the chain $\xi$.
Moreover, $f$ has a critical point $\bar{x}$ such that $f(\bar{x})=c$ and $m(\bar{x})\leq q\leq m^{*}(\bar{x})$.
\end{lemma}

\begin{remark}\label{2.01}
If $M$ is a finite dimensional Hilbert space, and
$f$ satisfies (C) condition instead of (PS) condition, the above theorem still holds, the proof is the same
as that of Theorem 4.1.7 in \cite{100} (see \cite{111} for results obtained under (C) condition).
\end{remark}

Recall that the functional $f$ satisfies the so called Cerami  condition ((C) condition for short) on $J\subset \mathbf{R}\bigcup \{\pm\infty\}$ if $\{z_{m}\}\subset M$
such that $f(z_{m})\rightarrow c\in J$ and
$(1+\| z_{m}\|)\|\nabla f(z_{m})\|\rightarrow 0$ as $m\rightarrow +\infty$ has a convergent subsequence.

\section{Proofs of the Main Results}
\setcounter{equation}{0}

{For simplicity, we first give a proof of Corollary \ref{1.000}.}

Define $f(z)=\frac{1}{2}\langle Az,z\rangle-\int_{0}^{\tau}H(t,z)\text{d}t$, $z\in E$,
by (H3), we have $f\in C^{2}(E,\mathbf{R})$. As usual, finding periodic solutions of the
system (\ref{1.1}) converts to looking for critical points of $f$.

Let $f_{m}=f|_{E_{m}}$, $X_{m}=E_{m}^{-}\bigoplus E^{0}$ and $Y_{m}=E_{m}^{+}$.
Now we check the conditions in Lemma \ref{2.1} for $f_{m}$ when $H$ satisfies (H1)-(H5).
The proofs are similar to those in \cite{110, 109}.

\begin{lemma}\label{3.11}
The functional $f$ satisfies (C)$^{*}$ condition with respect to
$\{E_{m} | m=1, 2, \cdots\}$, that is,
any sequence $\{z_{m}\}$ such that $z_{m}\in E_{m}$, $\{f_{m}(z_{m})\}$ is bounded and
$(1+\| z_{m}\|)\|\nabla f_{m}(z_{m})\|\rightarrow 0$ as $m\rightarrow +\infty$ has a convergent subsequence.
\end{lemma}

\begin{proof}
Let $\{z_{m}\}$ be such a sequence, we only need to prove $\{z_{m}\}$ is bounded.
Otherwise, we may suppose $\|z_{m}\|\rightarrow +\infty$ as $m\rightarrow +\infty$.

Note that for $z\in E_{m}$,  we have
$\nabla f_{m}(z)=P_{m}\nabla f(z)$. By using $\alpha_i+\beta_i=1$ and the integration by parts, there hold
$$f_{m}(z)-\langle \nabla f_{m}(z),V_{1}(z)\rangle=f(z)-\langle \nabla f(z),V_{1}(z)\rangle
=\int_{0}^{\tau}(H'_{z}(t,z)\cdot V_{1}(z)-H(t,z))\text{d}t,$$ where    $V_1(z)$ is defined  in (H2).
By (H2), we have
\begin{eqnarray}\label{3.16}
&&f_{m}(z_{m})-\langle \nabla f_{m}(z_{m}),V_{1}(z_{m})\rangle\nonumber\\
&=&\int_{0}^{\tau}(H'_{z}(t,z_{m})\cdot V_{1}(z_{m})-H(t,z_{m}))\text{d}t\nonumber\\
&\geq&c_{1}\|z_{m}\|_{L^{\beta}}^{\beta}-\tau c_{2}.
\end{eqnarray}
Hence $\{\|z_{m}\|_{L^{\beta}}\}$ is bounded.

By using the constants $\beta,\; \lambda$ defined in (H2), (H3), we set $p=\frac{2\beta+1}{2\lambda-1}$. It is obvious that $p>1$.
Take $q$ such that $\frac 1p+\frac 1q=1$.
It is easy to see  $\lambda-\frac{\beta}{p}=\frac{\lambda+\beta}{2\beta+1}<1$
and $2q(\lambda-\frac{\beta}{p})=\frac{\lambda+\beta}{\beta+1-\lambda}>1$.
By (H3), Lemma \ref{l2.1} and (\ref{3.16}), we have
\begin{eqnarray}\label{3.20}
&&\int_{0}^{\tau}|H'_{z}(t,z_{m})\cdot z^{\pm}_{m}|\text{d}t\nonumber\\
&\leq&c_{3}\int_{0}^{\tau}|z_{m}|^{\lambda}|z^{\pm}_{m}|\text{d}t+c_{4}\|z^{\pm}_{m}\|\nonumber\\
&=&c_{3}\int_{0}^{\tau}|z_{m}|^{\frac{\beta}{p}}|z|^{\lambda-\frac{\beta}{p}}|z^{\pm}_{m}|\text{d}t+c_{4}\|z^{\pm}_{m}\|\nonumber\\
&\leq&c_{3}\left(\int^{\tau}_{0}|z_{m}|^{\beta}\text{d}t\right)^{\frac{1}{p}}
\left(\int^{\tau}_{0}|z_{m}|^{\left(\lambda-\frac{\beta}{p}\right)q}
|z^{\pm}_{m}|^{q}\text{d}t\right)^{\frac{1}{q}}+c_{4}\|z^{\pm}_{m}\|\nonumber\\
&\leq&c_{5}\left(\int^{\tau}_{0}|z_{m}|^{\left(\lambda-\frac{\beta}{p}\right)2q}\text{d}t\right)^{\frac{1}{2q}}
\left(\int^{\tau}_{0}|z^{\pm}_{m}|^{2q}\text{d}t\right)^{\frac{1}{2q}}+c_{4}\|z^{\pm}_{m}\|\nonumber\\
&\leq&c_{6}\|z_{m}\|^{\lambda-\frac{\beta}{p}}\|z^{\pm}_{m}\|+c_{6}\|z^{\pm}_{m}\|,
\end{eqnarray}
where $c_i>0$ are suitable constants.

By (\ref{3.20}), we have
\begin{eqnarray}
\|\nabla f_{m}(z_{m})\|\cdot \|z_m^{\pm}\|
&\geq&\pm \langle\nabla f_{m}(z_{m}), z_{m}^{\pm}\rangle\nonumber\\
&=&\pm\langle Az_{m}, z_{m}^{\pm}\rangle
\mp\int_{0}^{\tau}H'_{z}(t,z_{m})\cdot z^{\pm}_{m}\text{d}t\nonumber\\
&\geq&\frac{2\pi}{\tau}\|z^{\pm}_{m}\|^{2}-
c_{6}\|z_{m}\|^{\lambda-\frac{\beta}{p}}\|z^{\pm}_{m}\|-c_{6}\|z^{\pm}_{m}\|.
\end{eqnarray}
We can suppose that there are only finitely many $z_m^{\pm}=0$. Dividing the two sides by $\|z_m\|\cdot \|z_m^{\pm}\|$, it implies $\frac{\|z^{\pm}_{m}\|}{\|z_{m}\|}\rightarrow 0$ as $m\rightarrow +\infty$.
By (\ref{3.16}), we see $\frac{|z_m^{0}|}{\|z_{m}\|}\rightarrow 0$, $m\rightarrow +\infty$.
From $\frac{\|z^{+}_{m}\|^{2}+\|z^{-}_{m}\|^{2}+\tau|z_m^{0}|^{2}}{\|z_{m}\|^{2}}=1$, we obtain a contradiction.
\end{proof}

We note that if $f$ satisfies (C)$^{*}$ condition on $E$,
then $f_{m}$ satisfies (C) condition on $E_{m}$.

There exists a constant $\eta>0$ such that $\tilde\sigma_{i}=\frac{\eta {\sigma}_{i}}{ {\sigma}_{i}+ {\tau}_{i}}\geq1$
and $\tilde\tau_{i}=\frac{\eta{\tau}_{i}}{ {\sigma}_{i}+ {\tau}_{i}}\geq1$.
For $\rho>0$ and $z=(p_1,\cdots, p_{n},q_1, \cdots, q_{n})\in E$, we set $B_{\rho}(z)=
(\rho^{\tilde\tau_{1}-1}p_{1},\cdots, \rho^{\tilde\tau_{n}-1}p_{n}, \rho^{\tilde\sigma_{1}-1}q_{1}\cdots, \rho^{\tilde\sigma_{n}-1}q_{n}).$
We note that $B_{\rho}$ is a linear  bounded and invertible operator and $\|B_\rho\|\leq1$, if $\rho\leq1$.

For $z=z^{+}+z^{0}+z^{-}\in E$, we have
\begin{equation}\label{3.1}
 \langle AB_{\rho}z,B_{\rho}z\rangle
=\rho^{\eta-2}\langle Az,z\rangle=\frac{2\pi}{\tau}\rho^{\eta-2}(\| z^{+}\|^{2}
-\| z^{-}\|^{2}).
\end{equation}

\begin{lemma}\label{3.12}
There exist $\mu\in (0,1)$ and $\delta>0$ independent of $m$ such that $\inf_{B_{\mu}(S_{m})}f_{m}\geq \delta$,
where $S_{m}=S\cap Y_{m}$ and $S=\{z\in E^{+} | \| z\|=\mu\}$.
\end{lemma}

\begin{proof} It suffices to show $\inf_{B_{\mu}(S)}f\geq \delta$.

By (H3) and (H4), for any $\varepsilon>0$, there exists $M_{\varepsilon}>0$ such that
\begin{equation}\label{3.2}
H(t,z)\leq\varepsilon\sum_{i=1}^{n}
 \left(|p_{i}|^{1+\frac{\sigma_{i}}{\tau_{i}}}+|q_{i}|^{1+\frac{\tau_{i}}{\sigma_{i}}}\right)
 +M_{\varepsilon}\sum_{i=1}^{n}
 \left(|p_{i}|^{1+\lambda}+|q_{i}|^{1+\lambda}\right),\ \ \ \ (t,z)\in \mathbf{R}\times\mathbf{R}^{2n}.
\end{equation}

By (\ref{3.2}), for $z=(p_{1}, \cdots, p_{n}, q_{1}, \cdots, q_{n})\in E$, $\| z\|=\mu$, we have
\begin{eqnarray}\label{3.3}
&&\int_{0}^{\tau}H(t,B_{\mu}z)\text{d}t\nonumber\\
&\leq&\varepsilon\sum_{i=1}^{n}
 \int_{0}^{\tau}\left(|\mu^{\tilde\tau_{i}-1}p_{i}|^{1+\frac{\sigma_{i}}{\tau_{i}}}
 +|\mu^{\tilde\sigma_{i}-1}q_{i}|^{1+\frac{\tau_{i}}{\sigma_{i}}}\right)\text{d}t
 +M_{\varepsilon}\sum_{i=1}^{n}
 \int_{0}^{\tau}\left(|\mu^{\tilde\tau_{i}-1}p_{i}|^{1+\lambda}+|\mu^{\tilde\sigma_{i}-1}q_{i}|^{1+\lambda}\right)\text{d}t\nonumber\\
&\leq&\varepsilon\sum_{i=1}^{n}
 \int_{0}^{\tau}\left(\mu^{(\tilde\tau_{i}-1)(1+\frac{\sigma_{i}}{\tau_{i}})}|z|^{1+\frac{\sigma_{i}}{\tau_{i}}}
 +\mu^{(\tilde\sigma_{i}-1)(1+\frac{\tau_{i}}{\sigma_{i}})}|z|^{1+\frac{\tau_{i}}{\sigma_{i}}}\right)\text{d}t\nonumber\\
&&+M_{\varepsilon}\sum_{i=1}^{n}
 \int_{0}^{\tau}\left(\mu^{(\tilde\tau_{i}-1)(1+\lambda)}|z|^{1+\lambda}
 +\mu^{(\tilde\sigma_{i}-1)(1+\lambda)}|z|^{1+\lambda}\right)\text{d}t\nonumber\\
&\leq&2\varepsilon\mu^{\eta}\sum_{i=1}^{n}C(\sigma_{i},\tau_{i})+
M_{\varepsilon}\mu^{\eta}\sum_{i=1}^{n}C(\lambda)
\left(\mu^{\tilde\tau_{i}(\lambda-\frac{\sigma_{i}}{\tau_{i}})}+\mu^{\tilde\sigma_{i}(\lambda-\frac{\tau_{i}}{\sigma_{i}})}\right)\nonumber\\
&\leq&c_{7}\mu^{\eta}\left[2\varepsilon+M_{\varepsilon}\sum_{i=1}^{n}
\left(\mu^{\tilde\tau_{i}(\lambda-\frac{\sigma_{i}}{\tau_{i}})}+
\mu^{\tilde\sigma_{i}(\lambda-\frac{\tau_{i}}{\sigma_{i}})}\right)\right],
\end{eqnarray}
where $C(\sigma_{i},\tau_{i}),\;C(\lambda)>0$ are the embedding constants.

By (\ref{3.1}) and (\ref{3.3}), for $z\in E^{+}$, $\| z\|=\mu$, we have
\begin{eqnarray}\label{3.4}
f(B_{\mu}z)&=&\frac{1}{2}\langle AB_{\mu}z,B_{\mu}z\rangle-\int_{0}^{\tau}H(t,B_{\mu}z)\text{d}t\nonumber\\
&\geq&\frac{\pi}{\tau}\mu^{\eta}-c_{7}\mu^{\eta}\left[2\varepsilon+M_{\varepsilon}\sum_{i=1}^{n}
\left(\mu^{\tilde\tau_{i}(\lambda-\frac{\sigma_{i}}{\tau_{i}})}+\mu^{\tilde\sigma_{i}(\lambda-\frac{\tau_{i}}{\sigma_{i}})}\right)\right].
\end{eqnarray}

Choose $\varepsilon>0$ and $0<\mu<1$ so small that $f(B_{\mu}z)\geq
\delta:=\frac{\pi}{3\tau}\mu^{\eta}$ for $z\in E^{+}$ and $\| z\|=\mu$.
Thus $\inf_{B_{\mu}(S)}f\geq \delta>0$.
\end{proof}

\begin{lemma}\label{l3.170} Set $Q=[\bar{B}_{\nu}\bigcap (E^{-}\bigoplus E^{0})]\bigoplus [0,\nu]e$ and $Q_{m}=Q\bigcap (X_{m}\bigoplus\text{span}\{e\})$,
where $\bar{B}_{\nu}=\{z\in E | \| z\|\leq \nu\}$. For any $\nu$ with $\nu>\mu>0$, we have
$B_{\nu}(\partial Q_{m})$ and $B_{\mu}(S_{m})$ homologically link.\end{lemma}

\begin{proof} Since $\nu>\mu>0$, then $\nu>\mu\| B_{\nu}^{-1} B_{\mu}\|=\mu\|B_{\frac{\mu}{\nu}}\|$ and
$PB^{-1}_{\mu}B_{\nu}: E^{-}\bigoplus E^{0}\rightarrow E^{-}\bigoplus E^{0}$ is linear bounded and invertible (see \cite{813, 110}),
where $P: E\rightarrow E^{-}\bigoplus E^{0}$ denotes the orthogonal projection.
Furthermore, by noting that $B_{i}(E_{m})\subset E_{m}$ $(i=\mu, \nu)$, and
$B_{i}|_{E_{m}}: E_{m}\rightarrow E_{m}$ is linear  bounded and invertible,
then $\tilde{P}_{m}(B_{\mu}|_{E_{m}})^{-1}B_{\nu}|_{E_{m}}: X_{m}\rightarrow X_{m}$
is linear  bounded and invertible, where
$\tilde{P}_{m}: E_{m}\rightarrow X_{m}$ is the orthogonal projection. From Lemma \ref{2.6}, we complete the proof.
\end{proof}

For $\varepsilon_{1}>0$ as in Lemma \ref{3.130}, we
set $\frac{2\pi}{\tau}\cdot\sqrt{2n}A^{-1}_{1}=\varepsilon_{1}\min_{1\leq i\leq n}
\left\{\left(\frac{\varepsilon_{1}}{\sqrt{2n}}\right)^{1+\frac{\sigma_{i}}{\tau_{i}}},
\left(\frac{\varepsilon_{1}}{\sqrt{2n}}\right)^{1+\frac{\tau_{i}}{\sigma_{i}}}\right\}$,
by (H5), there exists a constant $A_{2}>0$ such that
\begin{equation}\label{3.14}
H(t,z)\geq A_{1}\sum_{i=1}^{n}
 \left(|p_{i}|^{1+\frac{\sigma_{i}}{\tau_{i}}}+|q_{i}|^{1+\frac{\tau_{i}}{\sigma_{i}}}\right),
 \ \ \ \ (t,z)\in \mathbf{R}\times \mathbf{R}^{2n}\ \text{with}\ |z|\geq A_{2}.
\end{equation}

\begin{lemma}\label{3.170}
Choose $\nu>\frac{A_{2}}{\varepsilon_{1}}+1$, then
$f_{m}|_{B_{\nu}(\partial Q_{m})}\leq 0$.
\end{lemma}

\begin{proof} Since $\partial Q_{m}\subset \partial Q$, we show $f|_{B_{\nu}(\partial Q)}\leq 0$.

For $z\in\partial Q$, $z=se+z^{-}+z^{0}$, then $f(B_{\nu}z)\leq 0$. We show this in two cases.

Case 1. If $s=0$, by (H1) and (\ref{3.1}), we have $f(B_{\nu}z)\leq 0$.

Case 2. If $s\neq 0$, then $s=\nu$ and $\| z^{-}+z^{0}\|\leq \nu$
or $0\leq s\leq \nu$ and $\| z^{-}+z^{0}\|= \nu$.
In the two situations, we always have $\nu\leq\| z\|\leq 2\nu$.
We now consider two subcases.

Subcase 1. If $\| se+z^{0}\|<\| z^{-}\|$, so $\|se\|<\|z^-\|$,
by (H1) and (\ref{3.1}), then $f(B_{\nu}z)\leq 0$.

Subcase 2. If $\| se+z^{0}\|\geq\| z^{-}\|$,
set $\Omega_{z}=\{t\in [0,\tau]| |z(t)|\geq \nu\varepsilon_{1}\}$,
Lemma \ref{3.130} shows that measure $\Omega_{z}\geq\varepsilon_{1}$.
From definition, we have  \begin{equation}\frac{\sqrt{2n}}{\nu\varepsilon_{1}}|z(t)|\geq \sqrt{2n},\quad t\in \Omega_{z} \label{aaa} \end{equation}
  and
\be\label{bbb}|B_{\nu}z(t)|\geq |z(t)|\geq \nu\varepsilon_{1}>A_{2},\quad  t\in \Omega_{z}.\ee

From \eqref{aaa} and  Remark 1.4 of \cite{110}, there hold
\begin{eqnarray}\label{3.140}
&&\sum_{i=1}^{n}\left(\left|\frac{\sqrt{2n}}{\nu\varepsilon_{1}}p_{i}(t)\right|^{1+\frac{\sigma_{i}}{\tau_{i}}}
+\left|\frac{\sqrt{2n}}{\nu\varepsilon_{1}}q_{i}(t)\right|^{1+\frac{\tau_{i}}{\sigma_{i}}}\right)\nonumber\\
&\geq&\frac{1}{2n}\sum_{i=1}^{n}\left(\frac{\sqrt{2n}}{\nu\varepsilon_{1}}|p_{i}(t)|+
\frac{\sqrt{2n}}{\nu\varepsilon_{1}}|q_{i}(t)|\right)
\geq\frac{1}{2n}\frac{\sqrt{2n}}{\nu\varepsilon_{1}}|z(t)|,\quad t\in \Omega_{z} .
\end{eqnarray}

By (\ref{3.14}), \eqref{bbb} and (\ref{3.140}), we have
\begin{eqnarray}\label{3.150}
H(t,B_{\nu}z(t))&\geq& A_{1}\sum_{i=1}^{n}
 \left(|\nu^{\tilde\tau_{i}-1}p_{i}(t)|^{1+\frac{\sigma_{i}}{\tau_{i}}}
+|\nu^{\tilde\sigma_{i}-1}q_{i}(t)|^{1+\frac{\tau_{i}}{\sigma_{i}}}\right)\nonumber\\
&\geq&A_{1}\nu^{\eta}\sum_{i=1}^{n}
 \left(|\nu^{-1}p_{i}(t)|^{1+\frac{\sigma_{i}}{\tau_{i}}}+|\nu^{-1}q_{i}(t)|^{1+\frac{\tau_{i}}{\sigma_{i}}}\right)\nonumber\\
&\geq&A_{1}\nu^{\eta}\min_{1\leq i\leq n}\left\{\left(\frac{\varepsilon_{1}}{\sqrt{2n}}\right)^{1+\frac{\sigma_{i}}{\tau_{i}}},
\left(\frac{\varepsilon_{1}}{\sqrt{2n}}\right)^{1+\frac{\tau_{i}}{\sigma_{i}}}\right\}\cdot\nonumber\\
&&\sum_{i=1}^{n}\left(\left|\frac{\sqrt{2n}}{\nu\varepsilon_{1}}p_{i}(t)\right|^{1+\frac{\sigma_{i}}{\tau_{i}}}
+\left|\frac{\sqrt{2n}}{\nu\varepsilon_{1}}q_{i}(t)\right|^{1+\frac{\tau_{i}}{\sigma_{i}}}\right)\nonumber\\
&\geq&A_{1}\nu^{\eta}\min_{1\leq i\leq n}\left\{\left(\frac{\varepsilon_{1}}{\sqrt{2n}}\right)^{1+\frac{\sigma_{i}}{\tau_{i}}},
\left(\frac{\varepsilon_{1}}{\sqrt{2n}}\right)^{1+\frac{\tau_{i}}{\sigma_{i}}}\right\}
\frac{1}{2n}\frac{\sqrt{2n}}{\nu\varepsilon_{1}}|z(t)|\nonumber\\
&\geq&\frac{2\pi}{\tau\varepsilon_1}\nu^{\eta},\quad t\in \Omega_{z}.
\end{eqnarray}

By (H1), (\ref{3.1}) and (\ref{3.150}), we have
$f(B_{\nu}z)\leq \frac{2\pi}{\tau}\nu^{\eta}-\int_{\Omega_{z}}H(t,B_{\nu}z(t))\text{d}t\leq 0.$
\end{proof}

\begin{theorem}\label{3.160}
If $H$ satisfies (H1)-(H5), then there exists
a nonconstant solution $z$ of the system (\ref{1.1}) satisfying
\begin{equation}\label{3.161}
i_{\tau}(z)\leq n+1\leq i_{\tau}(z)+\nu_{\tau}(z).
\end{equation}
\end{theorem}

\begin{proof}
We follow the ideas of   \cite{102}. Now
Lemmas \ref{3.11}-\ref{3.170} show that all conditions of Lemma \ref{2.1} are satisfied  for $f_{m}$ (see Remark \ref{2.01}).
By  (\ref{3.1}) and (H1), we have $f|_{B_{\nu}(Q)}\leq \frac{2\pi}{\tau}\nu^{\eta}$. So $f_{m}$ has a critical point $z_{m}$ satisfying
\begin{equation}\label{3.162}
\delta\leq f_{m}(z_{m})\leq \frac{2\pi}{\tau}\nu^{\eta}\ \ \text{and}\ \
m(z_{m})\leq \ \text{dim} X_{m}+1\leq m^{*}(z_{m}).
\end{equation}

By Lemma \ref{3.11}, we may assume $z_{m}\rightarrow z\in E$
with $\delta\leq f(z)\leq \frac{2\pi}{\tau}\nu^{\eta}$ and $\nabla f(z)=0$.
By (H1), we see $z$ is a nonconstant solution of the system (\ref{1.1}).
Now we show that the critical point $z$ satisfies (\ref{3.161}).

Let $B$ be the operator for $B(t)=H_{zz}''(t,z(t))$ defined in Section 2, then we have
\begin{equation}\label{3.7}
\| f''(x)-f''(z)\|=\| f''(x)-(A-B)\|\rightarrow 0,\ \ \ \ \| x-z\|\rightarrow 0.
\end{equation}

Let $0<d<\| (A-B)^{\sharp}\|^{-1}$, by (\ref{3.7}), there exists a constant $\kappa>0$ such that
$$\| f''(x)-(A-B)\|<\frac{d}{3},\ \ \ \ x\in \bar{B}(z,\kappa):=\{z\in E| \| x-z\|\leq \kappa\}.$$

Then for $m$ large enough, we have
\begin{equation}\label{3.8}
\| f_{m}''(x)-P_{m}(A-B)P_{m}\|<\frac{d}{2},\ \ \ \ x\in \bar{B}(z,\kappa)\bigcap E_{m}.
\end{equation}

For $x\in \bar{B}(z,\kappa)\bigcap E_{m}$, Eq. (\ref{3.8}) implies that
\begin{eqnarray*}
\langle f_{m}''(x)u,u\rangle&\leq&\langle P_{m}(A-B)P_{m}u,u\rangle
+\| f_{m}''(x)-P_{m}(A-B)P_{m}\|\| u\|^{2}\\
&\leq& -\frac{d}{2}\| u\|^{2}<0,\ \ \ \ u\in M_{d}^{-}(P_{m}(A-B)P_{m})\backslash \{0\}.
\end{eqnarray*}

Thus,
\begin{equation}\label{3.9}
\text{dim}M^{-}(f_{m}''(x))\geq \text{dim}M_{d}^{-}(P_{m}(A-B)P_{m}),\ \ \ \ x\in \bar{B}(z,\kappa)\bigcap E_{m}.
\end{equation}

Similarly, we have
\begin{equation}\label{3.10}
\text{dim}M^{+}(f_{m}''(x))\geq \text{dim}M_{d}^{+}(P_{m}(A-B)P_{m}),\ \ \ \ x\in \bar{B}(z,\kappa)\bigcap E_{m}.
\end{equation}

By (\ref{3.162}), (\ref{3.9}) and (\ref{3.10}), for $m$ large enough, Lemma \ref{2.2} shows that
\begin{eqnarray*}
\frac{1}{2}\text{dim}E_{m}+n+1&=&\text{dim}X_{m}+1\geq m(z_{m})\\
&\geq&\text{dim}M_{d}^{-}(P_{m}(A-B)P_{m})\\
&=&\frac{1}{2}\text{dim}E_{m}+i_{\tau}(z)
\end{eqnarray*}
and
\begin{eqnarray*}
\frac{1}{2}\text{dim}E_{m}+n+1&=&\text{dim}X_{m}+1\leq m^{*}(z_{m})\\
&\leq&\text{dim}M_{d}^{-}(P_{m}(A-B)P_{m})+M_{d}^{0}(P_{m}(A-B)P_{m})\\
&=&\frac{1}{2}\text{dim}E_{m}+i_{\tau}(z)+\upsilon_{T}(z).
\end{eqnarray*}

The above two estimates show that (\ref{3.161}) holds.
\end{proof}

\begin{remark}\label{7.1}
Under either the conditions of Theorem \ref{1.20} or Theorem \ref{1.00},
the conclusion of Theorem \ref{3.160} still holds (see Remarks below).
\end{remark}

\textbf{Proof of {Corollary \ref{1.000}}.}
The proof is the same as that in \cite{102}. For readers convenience' we give the details here.

Since $H$ is $k\tau$-periodic, by Theorem \ref{3.160},
the system \ref{1.1} possesses a nonconstant $k\tau$-periodic solution $z_{k}$ satisfying
\begin{equation}\label{3.169}
i_{k\tau}(z_{k})\leq n+1\leq i_{k\tau}(z_{k})+\nu_{k\tau}(z_{k}).
\end{equation}

If $z_{k}$ and $z_{pk}$ are not geometrically distinct, by definition, there exist integers
$l$ and $m$ such that $l\ast z_{k}=m\ast z_{pk}$. By Proposition \ref{2.11}, we have
$i_{k\tau}(l\ast z_{k})=i_{k\tau}(z_{k})$, $\nu_{k\tau}(l\ast z_{k})=\nu_{k\tau}(z_{k})$ and
$i_{pkT}(m\ast z_{pk})=i_{pkT}(z_{pk})$, $\nu_{pkT}(m\ast z_{pk})=\nu_{pkT}(z_{pk})$.

Eq. (\ref{3.169}) shows that $i_{pkT}(z_{pk})\leq n+1$ and $i_{k\tau}(z_{k})+\nu_{k\tau}(z_{k})\geq n+1$.
Proposition \ref{2.10} shows that $p-n\leq n+1$ contradicting with the assumption $p>2n+1$.
Hence if $p>2n+1$, then
$z_{k}$ and $z_{pk}$ are geometrically distinct.

If all $z_{k}$ are non-degenerate, then $\nu_{k\tau}(z_{k})=0$ and $i_{k\tau}(z_{k})= n+1$ for $k\in \mathbf N$.
Proposition \ref{2.12} shows that $p+n\leq n+1$, so we get $p=1$.
Hence
$z_{k}$ and $z_{pk}$ are geometrically distinct when $p>1$. We complete the proof of Theorem \ref{1.20}. \hfill$\Box$

From Remark \ref{7.1}, we see the proofs of Theorem \ref{1.20}
and Theorem \ref{1.00} are similar to the proof of Corollary \ref{1.000}.

\begin{remark}
Under the conditions of Theorem \ref{1.20},
the conclusion of Theorem \ref{3.160} still holds.
\end{remark}

Indeed, for any $K>0$, we take a cut-off function
defined by
\[
   \chi(s)=\left\{
   \begin{array}{cc}
   \begin{aligned}
   1,\ \ \ &\mbox{$0\leq s\leq K$,}\\
   0,\ \ \ &\mbox{$s\geq K+1$,}
   \end{aligned}
  \ \ \ \  \mbox{and}\ \  \chi'|_{(K,K+1)}<0.
   \end{array}
   \right.
\]
We set $\gamma=\max_{1\leq i\leq n} \{\frac{\alpha_{i}}{\beta_{i}}, \frac{\beta_{i}}{\alpha_{i}},
\frac{\xi_{i}}{\eta_{i}}, \frac{\eta_{i}}{\xi_{i}},
\frac{{\sigma}_{i}}{{\tau}_{i}}, \frac{{\tau}_{i}}{{\sigma}_{i}}, \beta-1\}$.  Choosing $\lambda_{0}\in  (\gamma, 1+\beta)$
 and
$$C_{K}\geq \max \bigg\{
\max\limits_{^{\ \ \ \ t\in \textbf{R}}_{K\leq |z|\leq K+1}}{\frac{H(t,z)}{|z|^{\lambda_{0}+1}}},
\ \ \frac{c_{1}}{\min_{1\leq i\leq n}\{\alpha_{i}\lambda_{0}-\beta_{i},
\beta_{i}\lambda_{0}-\alpha_{i}\}},\ \
A_{1}\bigg\},$$
where $A_{1}$ is defined in (\ref{3.14}), $c_1,\;\alpha_i,\;\beta_i$ are defined in (H2). For $(t,z)\in \mathbf{R}\times \mathbf{R}^{2n}$, we set
$$H_{K}(t,z)=\chi(|z|)H(t,z)+(1-\chi(|z|))C_{K}|z|^{\lambda_{0}+1}.$$

If $K>0$ is large enough,
it is easy to show that $H_{K}$ satisfies (H2) and (H3)$'$ with the constants independent of $K$ (see \cite{109}).
The modified function $H_{K}$ also satisfies (H1), (H3)-(H5).

Let $f_{K}(z)=\frac{1}{2}\langle Az,z\rangle-\int^{\tau}_{0}H_{K}(t,z)\text{d}t$,
then $f_{K}\in C^{2}(E,\textbf{R})$.

By the choice of $\lambda_{0}$, there exists a constant $A_{2}>0$ such that
$$|z|^{\lambda_{0}+1}\geq\sum_{i=1}^{n}
 \left(|p_{i}|^{1+\frac{\sigma_{i}}{\tau_{i}}}+|q_{i}|^{1+\frac{\tau_{i}}{\sigma_{i}}}\right),\ \ \ |z|\geq A_{2},$$
then we have
\begin{equation*}
H_{K}(t,z)\geq A_{1}\sum_{i=1}^{n}
 \left(|p_{i}|^{1+\frac{\sigma_{i}}{\tau_{i}}}+|q_{i}|^{1+\frac{\tau_{i}}{\sigma_{i}}}\right),
 \ \ \ \ (t,z)\in \mathbf{R}\times \mathbf{R}^{2n}\ \text{with}\ |z|\geq A_{2}.
\end{equation*}

In all the arguments before, we replace $H$, $\lambda$ and $f$  by $H_{K}$, $\lambda_{0}$ and $f_{K}$ respectively,
we see $f_{K}$ possesses a critical point $z_{K}$ satisfying $0<\delta_K<f_{K}(z_{K})\leq \frac{2\pi}{\tau}\nu^{\eta}$ and
$i_{\tau}(z_{K})\leq n+1\leq i_{\tau}(z_{K})+\nu_{\tau}(z_{K}).$

By (H3)$'$, it is easy to prove that $z=z_{K}$ is independent of $K$ and a $\tau$-periodic nonconstant
solution of the system (\ref{1.1}) for $K$  large enough (see \cite{109}).

\begin{remark}
Under the conditions of Theorem \ref{1.00},
the conclusion of Theorem \ref{3.160} still holds.
\end{remark}

In fact, we also take the cut-off function $\chi\in C^{\infty}([0,+\infty), \mathbf{R})$   as before.

For $R$ defined in (C3), we set
 $R(K)\geq\max\limits_{^{\ \ \ \ t\in \textbf{R}}_{K\leq |z|\leq K+1}}
\frac{H(t,z)}{\sum_{i=1}^{n}(|p_{i}|^
{\varphi_{i}}+|q_{i}|^{\psi_{i}})}$ $(K>R)$ and define
$$H_{K}(t,z)=\chi(|z|)H(t,z)+(1-\chi(|z|))R(K)
\sum_{i=1}^{n}\left(|p_{i}|^{\theta^{-1}\varphi_{i}}+|q_{i}|^{\theta^{-1}\psi_{i}}\right),\;\;(t,z)\in {\mathbf R}\times {\mathbf R}^{2n}.$$
Then $H_{K}$ satisfies (C2) and (C3) with the constants independent of $K$ (see \cite{813})
if $R(K)$ and $R$ are large enough.

Let $f_{K}(z)=\frac{1}{2}\langle Az,z\rangle-\int^{\tau}_{0}H_{K}(t,z)\text{d}t$,
then $f_{K}\in C^{2}(E,\textbf{R})$.
It is easy to show that $f_{K}$ satisfies (PS)$^{*}$ condition (see \cite{813}). By the definition of $H_{K}$,
we can choose $\lambda_1>\max_{1\leq i\leq n}\{\frac{\varphi_{i}}{\psi_{i}}, \frac{\psi_{i}}{\varphi_{i}}\}$ such that $H_{K}$
satisfies (H4), then $f_{K}|_{B_{\mu(K)}(S)}\geq\delta_{K}>0$ (see Lemma \ref{3.12}).

From \cite{813}, we know that there exist constants $d_{1}>0$ and $d_{2}>0$ such that
\begin{eqnarray*}
H_{K}(t,z)&\geq& d_{1}\sum_{i=1}^{n}
\left(|p_{i}|^{\theta^{-1}\varphi_{i}}+|q_{i}|^{\theta^{-1}\psi_{i}}\right)-d_{2}\\
&=&d_{1}\sum_{i=1}^{n}
\left(|p_{i}|^{\theta^{-1}(1+\frac{\varphi_{i}}{\psi_{i}})}
+|q_{i}|^{\theta^{-1}(1+\frac{\psi_{i}}{\varphi_{i}})}\right)-d_{2},
 \ \ \ \ (t,z)\in \mathbf{R}\times \mathbf{R}^{2n},
\end{eqnarray*}
which indicates an inequality similar to (\ref{3.14}).

\section{The case: $H$ contains a quadratic term}
\setcounter{equation}{0}

Now we consider the case where $H(t,z)=\frac{1}{2}(\hat{B}(t)z,z)+\hat{H}(t,z)$. The proof of the following results are
 similar to that of {Theorems \ref{1.20}, \ref{1.00} and Corollary \ref{1.000}}, we only state the  results.

We set $\omega=\max_{t\in \mathbf{R}}|\hat{B}(t)|$ and suppose $H(t,z)\geq 0,\;(t,z)\in \mathbf{R}\times\mathbf{R}^{2n}$.

\begin{theorem}\label{t5.1}
{Suppose $\hat{H}$ satisfies (H1), (H2), (H3)$'$, (H4), (H5) and $\hat{B}(t)$ satisfies
$(\hat{B}(t)z,z)=2(\hat{B}(t)z,V_{1}(z)), (t,z)\in \mathbf{R}\times\mathbf{R}^{2n}$} and

(H6)  $\hat{B}(t)$ is a $\tau$-periodic, symmetric and continuous matrix function and satisfies
$$(\hat{B}(t)z,z)=2(\hat{B}(t)z,V_{2}(z)),\ \ \ \ (t,z)\in \mathbf{R}\times\mathbf{R}^{2n}.$$
We also require there exists an unbounded sequence $\{\varrho_{m}\}\subset (0,+\infty)$ with $\inf_{m}\varrho_{m}=0$ such that
\begin{equation*}\label{1.0}
(\hat{B}(t)B_{\varrho}z,B_{\varrho}z)=\varrho^{\eta-2}(\hat{B}(t)z,z),\ \ (t,z)\in \mathbf{R}\times\mathbf{R}^{2n}
\end{equation*}
holds for $\varrho\in\{\varrho_{m}\}$,
where $B_{\varrho}z=(\varrho^{\tilde\tau_{1}-1}p_{1}, \cdots, \varrho^{\tilde\tau_{n}-1}p_{n},
\varrho^{\tilde\sigma_{1}-1}q_{1}, \cdots, \varrho^{\tilde\sigma_{n}-1}q_{n})$, $\varrho>0$
with $\eta, \tilde\sigma_{i}, \tilde\tau_{i}$ defined as in Section 3.

\noindent Then for each integer $k\geq 1$ and $k<\frac{2\pi}{\omega \tau}$,
the system (\ref{1.1}) possesses a $kT$-periodic nonconstant solution $z_{k}$ such that
$z_{k}$ and $z_{pk}$ are geometrically distinct provided $p>2n+1$ and $pk<\frac{2\pi}{\omega \tau}$.
If all $z_{k}$ are non-degenerate, then
$z_{k}$ and $z_{pk}$ $(p>1)$ are geometrically distinct.
\end{theorem}

Note that (H6) is satisfied if $b_{ij}(t)=0$ whenever $|i-j|\neq n$.
If $\alpha_{i}=\beta_{i}=2$,
$\sigma_{i}=\tau_{i}=1$ $(i=1, 2, \cdots, n)$, then
$\hat{B}(t)$ is just a $\tau$-periodic, symmetric and continuous matrix function.

Similarly we have the following results.

\begin{corollary}\label{5.2}
{Replace (H3)$'$ with (H3)}, then we have the same results as in Theorem \ref{t5.1}.
\end{corollary}

\begin{theorem}\label{5.3}
Suppose $\hat{H}$ satisfies (C1)-(C4) and $\hat{B}(t)$ satisfies

(C5)  $\;\hat{B}(t)$ is a $\tau$-periodic, symmetric and continuous matrix function and satisfies
$$(\hat{B}(t)z,z)=2(\hat{B}(t)z,V_{3}(z)),\ \ \ \ (t,z)\in \mathbf{R}\times\mathbf{R}^{2n}.$$
Moreover,
set $\hat{\eta}=\max_{1\leq i\leq n}\{\varphi_{i}+\psi_{i}\}$,
$\hat{\sigma}_{i}=\frac{\hat{\eta}}{\varphi_{i}+\psi_{i}}\varphi_{i}$ and
$\hat{\tau}_{i}=\frac{\hat{\eta}}{\varphi_{i}+\psi_{i}}\psi_{i}$,
we require that there exists an unbounded sequence $\{\varrho_{m}\}\subset (0,+\infty)$ with $\inf_{m}\varrho_{m}=0$ such that
\begin{equation*}\label{1.02}
(\hat{B}(t)B_{\varrho}z,B_{\varrho}z)=\varrho^{\hat{\eta}-2}(\hat{B}(t)z,z),\ \ (t,z)\in \mathbf{R}\times\mathbf{R}^{2n},
\end{equation*}
holds for $\varrho\in\{\varrho_{m}\}$,
where $B_{\varrho}z=(\varrho^{\hat\tau_{1}-1}p_{1}, \cdots, \varrho^{\hat\tau_{n}-1}p_{n},
\varrho^{\hat\sigma_{1}-1}q_{1}, \cdots, \varrho^{\hat\sigma_{n}-1}q_{n})$, $\varrho>0$. Then we have the same results as in Theorem \ref{t5.1}.
\end{theorem}

\begin{remark}\label{5.4}
If $\varphi_{i}=\psi_{i}$ $(i=1, 2, \cdots, n)$, then
(C1)-(C5) are the conditions in \cite{102} with the difference that
we require $H(t,z)=\frac{1}{2}(\hat{B}(t)z,z)+\hat{H}(t,z)\geq 0$ instead of
$\Hat{B}(t)$ being semi-positive-definite. Thus
Theorem \ref{5.3} generalizes the theorems in \cite{102} in the semi-positive-definite case.
\end{remark}

For $z, \zeta\in E$,
define $\langle Bz, \zeta\rangle=\int^{\tau}_{0}(\hat{B}(t)z,\zeta)\text{d}t$,
then $B$ is a linear bounded and self-adjoint operator on $E$ and
$|\langle Bz, z\rangle|\leq\omega \|z\|^{2}.$

\begin{remark}
The key point of the proof of {Corollary \ref{5.2}} is that if (H6) holds, then we have $f(z)-f'(z)V_{1}(z)
=\int_{0}^{\tau}(\hat{H}'_{z}(t,z)\cdot V_{1}(z)-\hat{H}(t,z))\text{d}t$ and
$\langle BB_{\rho}z,B_{\rho}z\rangle=\rho^{\eta-2}\langle Bz,z\rangle$,
where $z\in E$, and $f$, $\eta$ and $B_{\rho}$ $(\rho>0)$ are defined in Section 3.
Note that $H$ satisfies (H3) if $\hat{H}$ does. Then the proof of (C)$^{*}$ condition is the same as that of Lemma \ref{3.11}.
We can define $B_{\mu}$ for small $\mu\in \{\varrho_{m}\}$ and $B_{\nu}$ for large $\nu\in \{\varrho_{m}\}$
as in Section 3. So the arguments can be applied to the current case.
\end{remark}

The first equation in (H6) implies that
$\hat{B}(t)=\hat{B}(t)V_{1}+V_{1}\hat{B}(t)$, $t\in \mathbf{R}$. For $\beta$ in (H2), we require that
$\beta\geq2$, so there exists $c_{\beta}>0$ such that
$\|z\|_{L^{\beta}}\geq c_{\beta}\|z\|_{L^{2}}$, $z\in L^{\beta}(S_{\tau},\mathbf{R}^{2n})$. So similarly we have the following result.

\begin{theorem}\label{5.10}
Suppose $\hat{H}$ satisfies (H1)-(H5) and $\hat{B}(t)$ satisfies

(H6)$'$  $\hat{B}(t)$ is a $\tau$-periodic, symmetric and continuous matrix function with
$\max_{t\in \mathbf{R}}|\hat{B}(t)-\hat{B}(t)V_{1}-V_{1}\hat{B}(t)|<c_{1}c_{\beta}$
and satisfies
$$\displaystyle\limsup_{\varrho\rightarrow0^{+}}(\hat{B}(t)B_{\varrho}z,B_{\varrho}z)\leq\omega_{1}\varrho^{\eta-2}\ \ and\ \
\displaystyle\liminf_{\varrho\rightarrow+\infty}(\hat{B}(t)B_{\varrho}z,B_{\varrho}z)\geq\omega_{2}\varrho^{\eta-2}$$
hold uniformly for $(t,z)\in \mathbf{R}\times\mathbf{R}^{2n}$ and $|z|=1$,
where $c_{1}$ is as in (H2) and $\omega_{1}, \omega_{2}\geq0$.

\noindent Then we have the same results as in Theorem \ref{t5.1}.
\end{theorem}

For $\theta, \varphi_{i}, \psi_{i}$ $(i=1, 2, \cdots, n)$ in (C2), we require that
$\gamma\geq2\theta$, so there exists $c_{\gamma}>0$ such that
$\|z\|_{L^{\gamma}}\geq c_{\gamma}\|z\|_{L^{2}}$, $z\in L^{\gamma}(S_{\tau},\mathbf{R}^{2n})$,
where $\gamma=\varphi_{i}, \psi_{i}$ $(i=1, 2, \cdots, n)$. So similarly we also have the following result.

\begin{theorem}\label{5.11}
Suppose $H$ satisfies a condition similar to (C3),
$\hat{H}$ satisfies (C1)-(C5) and $\hat{B}(t)$ satisfies

(C5)$'$  $\hat{B}(t)$ is a $\tau$-periodic, symmetric and continuous matrix function with
$\max_{t\in \mathbf{R}}|\hat{B}(t)-\hat{B}(t)V_{3}-V_{3}\hat{B}(t)|<c_{1}\min_{1\leq i\leq n}\{c_{\varphi_{i}}, c_{\psi_{i}}\}$
and satisfies
$$\displaystyle\limsup_{\varrho\rightarrow0^{+}}(\hat{B}(t)B_{\varrho}z,B_{\varrho}z)\leq\omega_{3}\varrho^{\hat{\eta}-2}\ \ and\ \
\displaystyle\liminf_{\varrho\rightarrow+\infty}(\hat{B}(t)B_{\varrho}z,B_{\varrho}z)\geq\omega_{4}\varrho^{\hat{\eta}-2}$$
hold uniformly for $(t,z)\in \mathbf{R}\times\mathbf{R}^{2n}$ and $|z|=1$,
where $\hat{\eta}$ and $B_{\varrho}$ are as in Theorem \ref{5.3} and $\omega_{3}, \omega_{4}\geq0$.

\noindent Then we have the same results as in Theorem \ref{t5.1}.
\end{theorem}

\section{Minimal periodic solutions for the autonomous Hamiltonian systems}
\setcounter{equation}{0}

In this section, we consider the minimal periodic problem of the following autonomous Hamiltonian systems
\begin{equation}\label{5.1}
\left\{\begin{array}{ll}\dot z=JH'(z),\ \ \ \  z\in \mathbf{R}^{2n},\\
z(\tau)=z(0).\end{array}\right.
\end{equation}

We say that $(z,\tau)$ is a minimal periodic solution of \eqref{5.1} if $z$ solves the problem \eqref{5.1} with $\tau$ being the minimal period of $z$.

As shown in \cite{112}, we can also obtain minimal periodic solutions for the autonomous Hamiltonian systems \eqref{5.1}.

From Theorem \ref{3.160} and Remark \ref{7.1}, we see
that for any $\tau>0$ the Hamiltonian system (\ref{5.1}) possesses a nontrivial $\tau$-periodic solution $(z, \tau)$ satisfying $i_{\tau}(z)\leq n+1$
 provided the function $H$ satisfying the conditions in one of the {Theorems \ref{1.20}, \ref{1.00} and Corollary \ref{1.000}}
with some necessary modifications in an obvious way (since the function $H$ does not depend on time $t$). In the following result we understand the statements of the conditions in this sense.

\begin{theorem}\label{5.010}
Suppose the autonomous Hamiltonian function $H(z)$ satisfies the conditions
in one of the {Theorems \ref{1.20}, \ref{1.00} and Corollary \ref{1.000}} and

{(H7) $H''_{zz}(z)$ is strictly positive for every $z\in \mathbf{R}^{2n}\setminus \{0\}$.}

\noindent Then $(z, \tau)$ is a minimal periodic  solution of the nonlinear Hamiltonian  system (\ref{5.1}).
\end{theorem}

{Note that the Hamiltonian function $H$ of (\ref{1.2}) satisfies (H7).}

\textbf{Proof of Theorem \ref{5.010}.} The proof is almost the same as that in \cite{112}.
For readers' convenience, we estimate the iteration number of the solution $(z,\tau)$ now.

Assume $(z,\tau)$ has minimal period $\frac{\tau}{k}$, i.e., its iteration number is $k\in \mathbf Z$. Since the nonlinear Hamiltonian system in \eqref{5.1} is autonomous and (H7) holds,
we have $\nu_{\frac{\tau}{k}}(z)\geq 1$ and $i_{\frac{\tau}{k}}(z)\geq n$ by Lemma \ref{2.121}.
From Lemma \ref{2.120}, we see $k=1$, that is, the solution $(z,\tau)$ has minimal period $\tau$.\hfill$\Box$

In his pioneer work \cite{Rab1}, P. Rabinowitz proposed a conjecture on whether a superquadratic Hamiltonian system possesses a periodic solution with a prescribed minimal period. This conjecture has been deeply studied by many mathematicians. We refer to \cite{121, 2, FE1, FEQQ1, FKW, 112, 105} for the original Rabinowitz's conjecture under some further conditions (for example the convex case).  For the minimal periodic problem of brake solution of Hamiltonian systems, we refer to \cite{Liu, LC2, Zh2014}. For the minimal periodic problem of $P$-symmetric solution of Hamiltonian systems, we refer to \cite{Liu1, LT2}. Up to our knowledge, Theorem \ref{5.010} is the first result on the minimal periodic problem of nonlinear Hamiltonian systems with anisotropic growth.

\end{document}